\documentclass[a4paper, 10pt, parskip=half]{scrartcl}

\usepackage[utf8]{inputenc}
\usepackage[T1]{fontenc}
\usepackage{fullpage}
\usepackage{amsmath,amsthm,amssymb,mathtools,dsfont}

\usepackage{hyperref}
\hypersetup{
  colorlinks=true,
  linkcolor=black,
  citecolor=black,
  urlcolor=blue,
  pdftitle={Classical and generalized solutions of an alarm-taxis model},
  pdfauthor={Fuest, Lankeit},
  pdfkeywords={},
  bookmarksopen=true,
}

\usepackage[numbers, compress]{natbib}

\numberwithin{equation}{section}
\newtheorem{thm}{Theorem}[section]
\newtheorem{lem}[thm]{Lemma}

\theoremstyle{definition}
\newtheorem{definition}[thm]{Definition}
\newtheorem{rem}[thm]{Remark}

\usepackage{newunicodechar}
\newunicodechar{μ}{\mu}
\newunicodechar{ψ}{\psi}
\newunicodechar{φ}{\varphi}
\newunicodechar{ε}{\varepsilon}
\newunicodechar{∞}{\infty}
\newunicodechar{Δ}{\Delta}
\newunicodechar{∇}{\nabla}
\newunicodechar{κ}{\kappa}
\newunicodechar{η}{\eta}
\newunicodechar{σ}{\sigma}
\newunicodechar{∂}{\partial}
\newunicodechar{ν}{\nu}
\newunicodechar{ζ}{\zeta}
\newunicodechar{χ}{\chi}

\newcommand{\Om}{\Omega}
\newcommand{\Ombar}{\overline \Omega}
\newcommand{\Ombarinf}{\Ombar \times [0, \infty)}

\newcommand{\Tmax}{T_{\mathrm{max}}}
\newcommand{\tmax}{\Tmax}
\newcommand{\f}[2]{\frac{#1}{#2}}
\newcommand{\ddt}{\frac{\mathrm{d}}{\mathrm{d}t}}
\newcommand{\io}{\int_{\Om}}

\newcommand{\intom}{\io}
\newcommand{\intntom}{\int_0^T \io}
\newcommand{\intnstom}{\int_0^t \io}
\newcommand{\intninfom}{\int_0^\infty \io}

\newcommand{\ol}{\overline}

\newcommand{\ur}[1]{\mathrm{#1}}
\newcommand{\ure}{\ur e}

\newcommand{\leb}[2][\Omega]{{L^{#2}(#1)}}
\newcommand{\sob}[3][\Omega]{{W^{#2, #3}(#1)}}
\newcommand{\con}[2][\Ombar]{{C^{#2}(#1)}}

\newcommand{\hp}{\hphantom}
\newcommand{\pe}{\mathrel{\hp{=}}}

\newcommand{\embed}{\hookrightarrow}

\newcommand{\eps}{\varepsilon}

\newcommand{\gt}{>}
\newcommand{\lt}{<}

\newcommand{\defs}{\coloneqq}
\newcommand{\sfed}{\eqqcolon}

\newcommand{\R}{\mathbb{R}}

\newcommand{\N}{\mathbb{N}}

\DeclareMathOperator{\supp}{supp}

\newcommand{\diff}{\mathrm d}
\newcommand{\dx}{\,\diff x}
\newcommand{\ds}{\,\diff s}

\newcommand{\dtau}{\,\diff \tau}

\newcommand{\ue}{u_\eps}
\newcommand{\uej}{u_{\eps_j}}
\newcommand{\uet}{u_{\eps t}}
\newcommand{\ve}{v_\eps}
\newcommand{\vej}{v_{\eps_j}}
\newcommand{\vet}{v_{\eps t}}
\newcommand{\we}{w_\eps}
\newcommand{\wej}{w_{\eps_j}}
\newcommand{\wet}{w_{\eps t}}
\newcommand{\une}{u_{0\eps}}
\newcommand{\vne}{v_{0\eps}}
\newcommand{\wne}{w_{0\eps}}
\newcommand{\wnej}{w_{0\eps_j}}

\newcommand{\sigmae}{\sigma_{\eps}}
\newcommand{\Tmaxe}{T_{\mathrm{max}, \eps}}
\newcommand{\tmaxe}{\Tmaxe}

\newcommand{\ra}{\rightarrow}
\newcommand{\sea}{\searrow}
\newcommand{\rh}{\rightharpoonup}
\newcommand{\wto}{\rightharpoonup}

\newcommand{\loc}{\mathrm{loc}}

\makeatletter
\renewenvironment{proof}[1][\proofname]{\par
  \pushQED{\qed}%
  \normalfont \topsep0\p@\relax
  \trivlist
  \item[\hskip\labelsep\scshape
  #1\@addpunct{.}]\ignorespaces
}{%
  \popQED\endtrivlist\@endpefalse
}
\makeatother

\usepackage{authblk}

\usepackage[dvipsnames]{xcolor}

\author[1]{Mario~Fuest\footnote{e-mail: fuest@ifam.uni-hannover.de, ORCID: \href{https://orcid.org/0000-0002-8471-4451}{\color{black}0000-0002-8471-4451}}}
\author[1]{Johannes~Lankeit\footnote{e-mail: lankeit@ifam.uni-hannover.de, ORCID: \href{https://orcid.org/0000-0002-2563-7759}{\color{black}0000-0002-2563-7759}}}

\affil[1]{Leibniz Universität Hannover \protect\\ Institut für Angewandte Mathematik \protect\\ Welfengarten 1, 30167 Hannover, Germany}

\title{Classical and generalized solutions of an alarm-taxis model}

\begin{document}
\date{}
\maketitle  
\begin{abstract}
\noindent
\textbf{Abstract.}
In bounded, spatially two-dimensional domains, the system
\begin{equation*}
  \left\lbrace\begin{alignedat}{3}
    u_t &= d_1 \Delta u &&                                    &&+ u(\lambda_1 - \mu_1 u - a_1 v - a_2 w), \\
    v_t &= d_2 \Delta v &&- \xi \nabla \cdot (v \nabla u)     &&+ v(\lambda_2 - \mu_2 v + b_1 u - a_3 w), \\
    w_t &= d_3 \Delta w &&- \chi \nabla \cdot (w \nabla (uv)) &&+ w(\lambda_3 - \mu_3 w + b_2 u + b_3 v),
  \end{alignedat}\right.
\end{equation*}
complemented with initial and homogeneous Neumann boundary conditions,
models the interaction between prey (with density $u$), predator (with density $v$) and superpredator (with density $w$), which preys on both other populations.
Apart from random motion and prey-tactical behavior of the primary predator, the key aspect of this system is that the secondary predator reacts to alarm calls of the prey,
issued by the latter whenever attacked by the primary predator.

We first show in the pure alarm-taxis model, i.e.\ if $\xi = 0$, that global classical solutions exist.

For the full model (with $\xi > 0$),
the taxis terms and the presence of the term $-a_2 uw$ in the first equation
apparently hinder certain bootstrap procedures, meaning that the available regularity information is rather limited.
Nonetheless, we are able to obtain global generalized solutions.
An important technical challenge is to guarantee strong convergence of (weighted) gradients of the first two solution components
in order to conclude that approximate solutions converge to a generalized solution of the limit problem.\\
\textbf{Key words:} alarm-taxis, predator-prey, prey-taxis, food-chain, classical solution, generalized solution\\
\textbf{Mathematics Subject Classification (MSC 2020):} 35K20, 92D40, 35Q92, 35A01, 35D99
\end{abstract}

\section{Introduction}
An essential step in understanding ecosystems is shedding light on the interactions of the participating species 
with one another or with the environment. The first among such interactions studied in mathematical models are predator-prey, competitive or symbiotic relations between two species, as modeled by Lotka--Volterra type systems \cite{lotka,volterra,lotka-volterra-book}, these three types differing from each other by the signs of their `reaction terms'. Incorporating spatial dependence admits the description of motion, usually in form of random motion \cite{cantrell_cosner_book,juengel} which can already alter the relative success of competing species \cite{cantrell_cosner_lou}. But also directed advances towards, e.g.\ stationary food sources (e.g.\ \cite{cantrell_cosner_lou_movement}) are treated (and shown to be advantageous). 
A particular of such sources may be given by individuals of prey, resulting in prey-taxis, as studied by Kareiva and Odell \cite{KareivaOdellSwarmsPredatorsExhibit1987} for ladybug beetles and goldenrod aphids 
and mathematically investigated (with regards to solvability in classical or weak settings) 
in, e.g.\ \cite{WuEtAlGlobalExistenceSolutions2016,JinWangGlobalStabilityPreytaxis2017, XiangGlobalDynamicsDiffusive2018,HeZhengGlobalBoundednessSolutions2015, TaoGlobalExistenceClassical2010,WinklerAsymptoticHomogenizationThreedimensional2017}.
On the opposite side of this phenomenon, the prey may react to the presence of predators by attempting to evade, leading to predator-taxis models (for a mathematical treatment, see \cite{WuEtAlDynamicsPatternFormation2018}). 

The combination of both types of taxis (predator-prey taxis, pursuit-evasion dynamics; \cite{TsyganovEtAlQuasisolitonInteractionPursuit2003,GoudonUrrutiaAnalysisKineticMacroscopic2016,TyutyunovEtAlMinimalModelPursuitEvasion2007}) leads to mathematically more challenging systems, due to their highly cross-diffusive structure. Nevertheless, some results on 
weak solvability \cite{TaoWinklerFullyCrossdiffusiveTwocomponent2021,TaoWinklerExistenceTheoryQualitative2022,FuestGlobalWeakSolutions2022} or stability of stationary states \cite{FuestGlobalSolutionsHomogeneous2020} are available. 

The step to more than two species can make a significant difference: Even in ODE models, suddenly chaotic behaviour can be observed \cite{hastings_powell}. Regarding mechanisms of interaction, it opens up further possibilities 
(besides additional sets of predator/prey relationships). For example, it becomes possible for prey to engage in anti-predator behaviour by way of attracting a secondary predator that attacks the primary predator, thus aiding the prey 
(even though, at the same time, it may still prey on the original prey).  An alarm call model capturing this behaviour has been suggested by \cite{HaskellBellModelBurglarAlarm2021}. We consider the following system (with $u$ denoting the density of prey, $v$ that of the predator und $w$ that of the superpredator): 
\begin{align}\label{system}
  \begin{cases}
    u_t = d_1 \Delta u                                     + f(u, v, w) & \text{in $\Omega \times (0, \infty)$}, \\
    v_t = d_2 \Delta v - \xi \nabla \cdot (v \nabla u)     + g(u, v, w) & \text{in $\Omega \times (0, \infty)$}, \\
    w_t = d_3 \Delta w - \chi \nabla \cdot (w \nabla (uv)) + h(u, v, w) & \text{in $\Omega \times (0, \infty)$}, \\
    \partial_\nu u = \partial_\nu v = \partial_\nu w = 0                & \text{on $\partial \Omega \times (0, \infty)$}, \\
    u(\cdot, 0) = u_0, v(\cdot, 0) = v_0, w(\cdot, 0) = w_0             & \text{in $\Omega$}
  \end{cases}  
\end{align}
with
\begin{align}
  f(u, v, w) &= u(\lambda_1 - \mu_1 u - a_1 v - a_2 w), \label{eq:intro:def_f}\\
  g(u, v, w) &= v(\lambda_2 - \mu_2 v + b_1 u - a_3 w), \label{eq:intro:def_g}\\
  h(u, v, w) &= w(\lambda_3 - \mu_3 w + b_2 u + b_3 v), \label{eq:intro:def_h}
\end{align}
where
\begin{align}\label{eq:intro:params}
  d_i, \chi, \lambda_i, \mu_i, a_i, b_i > 0 \text{ for  } i \in \{1, 2, 3\} \quad \text{and} \quad \xi \ge 0.
\end{align}
The most notable feature of \eqref{system} is the alarm-taxis term, $-χ∇\cdot(w∇(uv))$. Compared to other taxis models, the `signal' therein consists in the nonlinear expression $uv$ arising from the clash of prey and primary predator. Another key challenge for the mathematical analysis stems from the fact that the directed motion of the secondary predator is (at least partially) influenced by another component which itself undergoes some taxis, a feature \eqref{system} shares with certain food-chain models (e.g.\ \cite{JinWangWu,JinLuZou,LiWang_Threespecies}) and forager-exploiter systems (cf.\ \cite{forager_exploiter_win,forager_exploiter_taowin,forager_exploiter_black,forager_exploiter_wangwang,forager_exploiter_liuzhuang,forager_exploiter_liwang}, for instance). 

For one-dimensional domains, global existence of classical solutions to \eqref{system} was established in \cite{HaskellBellModelBurglarAlarm2021}, where also pattern formation was discussed. 

In two-dimensional settings, replacing the summand $w(1-w)$ in $h$ by $w(1-w^{\sigma})$ for any $\sigma > 1$ ensures global  boundedness of solutions \cite{LiWangGlobalDynamicsPredatorprey2023}, but for $\sigma=1$, a crucial gradient estimate no longer works (cf.\ \cite[Remark 1.2]{LiWangGlobalDynamicsPredatorprey2023}). If, instead, the term $uw$ in $f$ and $h$ is replaced by a functional response of the form $\f{uw}{u+w}$, \cite{JinEtAlGlobalSolvabilityStability2023} proves global boundedness of classical solutions. 
In both of these settings, smallness of the taxis coefficients $\xi$ and $\chi$ leads to exponential convergence of solutions to a constant steady state.

Higher-dimensional domains have been treated in \cite{ZhangXuXin} (it seems worth noting that also three-dimensional domains are biologically relevant; for an overview of alarm signals in aquatic ecosystems, see e.g.\ \cite{ferrari_wisenden_chivers}) and it was shown there that the system admits globally bounded classical solutions as long as the coefficients of the logistic terms are sufficiently large.

At first glance, the main difference between \eqref{system} and the three-species food-chain model in \cite{JinWangWu} seems to be that between alarm-taxis and prey-taxis in the third equation. Actually, however, at least for questions of global solvability, this inclusion of another function in the taxis term does not constitute an overwhelmingly large difference -- after all, this function is rather soon seen to be bounded. A larger change is the interaction between first and third trophic level, i.e.\ the insertion of the $-uw$ term in the first equation.
Its absence enables the authors of \cite{JinWangWu} to apply suitable testing procedures eventually obtain uniform-in-time $\leb4$ bounds for $\nabla u$,
which, as $4$ is larger than the spatial dimension of $\Omega$, allows for semigroup arguments yielding boundedness of $v$, even without $g$ containing logistic terms. 
Biologically, this term models a direct predator-prey relationship between secondary predator and prey
(that is, its presence turns ``the prey calls for help'' into ``the prey calls for help, despite detrimental effects'').

The full coupling term $uw$ is contained in the alarm-taxis model in \cite{LiWangGlobalDynamicsPredatorprey2023} and the food-chain model in \cite{XuYangXin}, but both results require stronger than quadratic absorption terms in the equation for $w$.

A weaker form of this effect is considered for instance in the above-mentioned \cite{JinEtAlGlobalSolvabilityStability2023},
where an alarm-taxis system with `ratio-dependent' functional response ($\f{uw}{u+w}$ instead of $uw$ in \eqref{system}) is studied.
However, although such a term can be biologically motivated, mathematically the change from $uw$ is rather drastic.
After all, it enables the estimate $|\f{uw}{u+w}|\le u$ and due to a~priori estimates for $u$ and $v$ (see Lemma~\ref{lm:ue_bdds} and Lemma~\ref{lm:ve_bdds}), 
the function $f(u, v, w)$ is readily seen to be uniformly in time bounded in $\leb2$, which in turn implies an $L^\infty$-$L^4$ bound for $\nabla u$ and then, as above, boundedness of $v$
(cf.\ \cite[Lemma~3.3 and Lemma~3.5]{JinEtAlGlobalSolvabilityStability2023}).
In contrast, for \eqref{system}, we are only able to bound $\nabla u$ in $L^\infty$-$L^2$, which appears to be (barely) insufficient to start any bootstrap procedures for $v$.

The results of the present article are twofold. First, we consider a pure alarm-taxis system without prey-tactic effects on the intermediate trophic level. Here it turns out that classical solutions are global.
\begin{thm}\label{th:xi=0}
  Suppose that $\Omega \subset \R^2$ is a smooth, bounded domain and that \eqref{eq:intro:def_f}--\eqref{eq:intro:params} are fulfilled, and let
  \begin{align}\label{eq:intro:init}
  0 \le u_0 \in \bigcup_{\theta > 2} \sob1\theta, \quad
  0 \le v_0 \in \bigcup_{\theta > 2} \sob1\theta
  \quad \text{and} \quad
  0 \le w_0 \in \con0.
  \end{align}
  If $\xi = 0$, then there exists a global classical solution $(u, v, w)$ of \eqref{system}.
\end{thm}

For the full model, including prey-taxis of the primary and alarm-taxis of the secondary predator, we obtain global generalized solutions under mere positivity assumptions (and without largeness conditions) on the parameters.

\begin{thm}\label{th:xi>0}
  Suppose that $\Omega \subset \R^2$ is a smooth, bounded domain and that \eqref{eq:intro:def_f}--\eqref{eq:intro:params} are fulfilled, and let $0 \le u_0 \in \sob12$, $0 \le v_0 \in \leb2$ and $0 \le w_0 \in \leb1$.
  If $\xi > 0$, then there exists a global generalized solution $(u, v, w)$ of \eqref{system} in the sense of Definition~\ref{def:sol_concept} below.
\end{thm}

The solution concept will be made precise in Definition~\ref{def:sol_concept}. It is a relative of renormalized solutions, as established in the context of the Boltzmann equation in \cite{diperna_lions} and used in context of chemotaxis models (see, e.g.\ the survey \cite{lanwin_facing_low_regularity} or \cite[Section~1.2]{FuestStrongConvergenceWeighted2023} for an overview).

\paragraph{Plan of the paper.}
We start by stating a local existence result and collecting first basic estimates in Section~\ref{sec:prelim}.
As observed at the beginning of Section~\ref{sec:xi=0}, for $\xi = 0$ the system \eqref{system} resembles a chemotaxis-consumption system with a logistic source term,
for which global existence of classical solutions in two-dimensional settings is to be expected.
That is, by choosing a proof which does not rely on intricate energy-type arguments based on cancellations of worrisome terms, we rather rapidly obtain Theorem~\ref{th:xi=0}.

The proof of Theorem~\ref{th:xi>0} is considerably more delicate.
Mainly due to the fact that $u$ and $uv$ are contained in the cross-diffusive terms in the second and third equation in \eqref{system}, respectively,
the available a~priori estimates for $u$, $v$ and $w$ get steadily worse.
While the quadratic dampening term in the (second and) third equation implies a uniform-in-time bound for $\intom |\nabla v|^2$,
the exponent $2$ does not exceed the spatial dimension and hence well-established semigroup arguments yielding boundedness of $v$ seem to be unavailable.
Nonetheless, in Subsection~\ref{lm:sol_concept_consistent} we are able to collect sufficient a~priori estimates
to conclude that a pointwise limit of solutions to approximate problems is sufficiently regular for all integral terms in  Definition~\ref{def:sol_concept} of generalized solutions to make sense.

However, as some of the integrands there are quadratic in $(\nabla u, \nabla v, \nabla w)$, the weak convergence obtained by boundedness in certain reflexive spaces turns out to be insufficient.
To overcome this issue in Subsection~\ref{sec:xi>0:eps_sea_0}, we crucially make use of the main results in \cite{FuestStrongConvergenceWeighted2023}
which inter alia assert \emph{strong} convergence of weighted gradients as long as the right-hand side of the considered parabolic equation converges weakly in space-time $L^1$.
This allows us prove in Subsection~\ref{sec:xi>0:proof_thm} that the limit $(u, v, w)$ is indeed a global generalized solution of \eqref{system}, i.e.\ to prove Theorem~\ref{th:xi>0}.

\section{Preliminaries}\label{sec:prelim}
As first step, we state a local existence result, including an extensibility criterion. The function $\sigma$ is included so that we can use the same lemma for both the classical solutions of Theorem~\ref{th:xi=0} and for solutions to the approximate system used during the proof of Theorem~\ref{th:xi>0} in Section~\ref{sec:xi>0}.

\begin{lem}\label{lm:local_ex}
  Let $\Omega \subset \R^2$ be a smooth, bounded domain, $f, g, h$ as in \eqref{eq:intro:def_f}--\eqref{eq:intro:def_h}, $d_i, \chi, \xi, \lambda_i, a_i, b_i$ as in \eqref{eq:intro:params},
  $\sigma \in C^\infty([0, \infty))$ and $0 \le u_0 \in \sob1\theta$, $0 \le v_0 \in \sob1\theta$, $0 \le w_0 \in \con0$ for some $\theta > 2$.
  Then there exists $\tmax = \tmax(u_0, v_0, w_0) \in (0, \infty]$ and a nonnegative, unique maximal classical solution
  \begin{align}\label{eq:local_ex:reg}
    (u, v, w) \in \big(C^{2, 1}(\Ombar \times [0, \tmax))\big)^3 \times C^0([0, \tmax); \sob1\theta \times \sob1\theta \times \con0)
  \end{align}
  of
  \begin{align}\label{prob:gen}
    \begin{cases}
      u_t = d_1 \Delta u                                              + f(u, v, w) & \text{in $\Omega \times (0, \tmax)$}, \\
      v_t = d_2 \Delta v - \xi \nabla \cdot (\sigma(v) \nabla u)      + g(u, v, w) & \text{in $\Omega \times (0, \tmax)$}, \\
      w_t = d_3 \Delta w - \chi \nabla \cdot (\sigma(w) \nabla (u v)) + h(u, v, w) & \text{in $\Omega \times (0, \tmax)$}, \\
      \partial_\nu u = \partial_\nu v = \partial_\nu w = 0                         & \text{on $\partial \Omega \times (0, \tmax)$}, \\
      u(\cdot, 0) = u_0, v(\cdot, 0) = v_0, w(\cdot, 0) = w_0                      & \text{in $\Omega$}.
    \end{cases}  
  \end{align}
  Moreover, $\tmax < \infty$ implies
  \begin{align}\label{eq:local_ex:ext}
    \limsup_{t \nearrow \tmax} \left( \|u(\cdot, t)\|_{\sob1q} + \|v(\cdot, t)\|_{\sob1q} + \|w(\cdot, t)\|_{\con0} \right) = \infty
    \qquad \text{ for all $q > 2$}.
  \end{align}
\end{lem}
\begin{proof}
  This follows by a straightforward adaptation of well-established fixed point arguments as employed for instance in \cite[Lemma 3.1]{BellomoEtAlMathematicalTheoryKeller2015}.
\end{proof}

Essentially because $f$, $g$ and $h$ model the interaction of a food chain, there is a positive linear combination of these functions which grows at most linearly.
Thus, as first yet very basic a~priori estimates for solutions of \eqref{prob:gen} we obtain uniform-in-time $L^1$ bounds and space-time $L^2$-$L^2$ bounds.
\begin{lem}\label{lm:u_v_w_l1}
  Let $\Omega \subset \R^2$ be a smooth, bounded domain, $f, g, h$ as in \eqref{eq:intro:def_f}--\eqref{eq:intro:def_h},
  $d_i, \chi, \xi, \lambda_i, a_i, b_i$ as in \eqref{eq:intro:params}, $T_0 > 0$ and $M > 0$.
  Then there exists $C > 0$ such that for any $\sigma \in C^\infty([0, \infty))$ with $0 \le \sigma \le \mathrm{id}$
  and any nonnegative $u_0 \in \bigcup_{\theta > 2} \sob1\theta$, $v_0 \in \bigcup_{\theta > 2} \sob1\theta$, $w_0 \in \con0$ with
  \begin{align*}
    \intom ( u_0 + v_0 + w_0 ) \le M,
  \end{align*}
  the solution $(u, v, w)$ of \eqref{prob:gen} given by Lemma~\ref{lm:local_ex} fulfils
  \begin{align}\label{eq:u_v_w_l1:l1}
    \sup_{t \in (0, T)} \left( \intom u(\cdot, t) + \intom v(\cdot, t) + \intom w(\cdot, t) \right) \le C
  \end{align}
  and
  \begin{align}\label{eq:u_v_w_l1:l2l2}
    \intntom u^2 + \intntom v^2 + \intntom w^2 \le C,
  \end{align}
  where $T \defs \min\{T_0, \tmax(u_0, v_0, w_0)\}$.
\end{lem}
\begin{proof}
  We choose $\eta_1, \eta_2 \in (0, 1)$ with $b_1 \eta_1 \le a_1$, $b_2 \eta_2 \le a_2$ and $b_3 \eta_2 \le a_3 \eta_1$
  and set $\lambda \defs \max\{\lambda_1, \lambda_2, \lambda_3\}$ and $\mu \defs \min \{\mu_1, \mu_2, \mu_3\}$.
  Then two applications of Jensen's inequality imply
  \begin{align}\label{eq:u_v_w_l1:testing}
          \ddt \intom (u + \eta_1 v + \eta_2 w )
    &=    \lambda \intom (u + \eta_1 v + \eta_2 w )
          - \mu \intom  (u^2 + \eta_1 v^2 + \eta_2 w^2 ) \notag \\
    &\pe  + \intom (-a_1 uv - a_2 uw + b_1\eta_1 uv - a_3\eta_1 vw + b_2 \eta_2 uw + b_3 \eta_2 vw ) \notag \\
    &\le  \lambda \intom (u + \eta_1 v + \eta_2 w )
          - \frac{\mu}{3} \intom ( u + \eta_1 v + \eta_2 w )^2 \notag \\
    &\le  \lambda \intom (u + \eta_1 v + \eta_2 w ) 
          - \frac{\mu}{3|\Omega|} \left(\intom (u + \eta_1 v + \eta_2 w ) \right)^2
  \end{align}
  in $[0, T)$.
  Thus, by an ODE comparison argument,
  \begin{align*}
        \intom (u + \eta_1 v + \eta_2 w )
    \le \max\left\{\intom (u_0 + \eta_1 v_0 + \eta_2 w_0 ), \frac{3\lambda |\Omega|}{\mu}\right\}
    \le \max\left\{M, \frac{3\lambda |\Omega|}{\mu}\right\}
    \sfed c_1
  \end{align*}
  in $(0, T)$.
  Moreover, integrating the first inequality in \eqref{eq:u_v_w_l1:testing} in time
  shows
  \begin{align*}
          \mu \int_0^T \intom (u^2 + \eta_1 v^2 + \eta_2 w^2 )
    &\le  c_1 + \lambda c_1 T,
  \end{align*}
  so that \eqref{eq:u_v_w_l1:l1} and \eqref{eq:u_v_w_l1:l2l2} hold with $C \defs \frac{\max\{c_1, \frac{1}{\mu}(c_1 + \lambda c_1 T)\}}{\min\{\eta_1, \eta_2\}}$.
\end{proof}

Next, we state two consequences of the Gagliardo--Nirenberg inequality.
\begin{lem}\label{lm:gni_comb}
  Let $\Omega \subset \R^2$ be a smooth, bounded domain.
  For all $\eps > 0$, there exists $C > 0$ such that 
  \begin{align}\label{eq:gni_comb:lady}
    \intom |\nabla \psi|^4 \le C \intom |\Delta \psi|^2 \intom |\nabla \psi|^2
  \end{align}
  and
  \begin{align}\label{eq:gni_comb:comb}
    \intom \varphi^2 |\nabla \psi|^2 \le \eps \intom |\nabla \varphi|^2 + C \left( \intom |\Delta \psi|^2 \intom |\nabla \psi|^2 + 1 \right) \intom \varphi^2
  \end{align}
  for all $(\varphi, \psi) \in X \defs \{\,(\tilde \varphi, \tilde \psi) \in \con1 \times \con2 \mid \partial_\nu \tilde \psi = 0 \text{ on } \partial \Omega\,\}$.
\end{lem}
\begin{proof}
  By the Gagliardo--Nirenberg inequality, elliptic regularity theory (cf.\ \cite[Theorem~19.1]{FriedmanPartialDifferentialEquations1976}) and the Poincar\'e inequality,
  there exist $c_1, c_2 > 0$ such that
  \begin{align*}
    \|\varphi\|_{\leb4}^2
    &\le c_1 \|\nabla \varphi\|_{\leb2} \|\varphi\|_{\leb2} + c_1 \|\varphi\|_{\leb2}^2
    \intertext{and}
    \|\nabla \psi\|_{\leb4}^2
    &\le c_1 \|\Delta \psi\|_{\leb2} \|\nabla \psi\|_{\leb2} + c_1 \|\nabla \psi\|_{\leb2}^2
     \le c_2 \|\Delta \psi\|_{\leb2} \|\nabla \psi\|_{\leb2}
  \end{align*}
  for all $(\varphi, \psi) \in X$.
  Thus, by Hölder's and Young's inequality there is some $c_3 > 0$ with
  \begin{align*}
          \|\varphi \nabla \psi\|_{\leb2}^2
    &\le  \|\varphi\|_{\leb4}^2 \|\nabla \psi\|_{\leb4}^2 \\
    &\le  c_1 c_2 \left( \|\nabla \varphi\|_{\leb2} \|\varphi\|_{\leb2} + \|\varphi\|_{\leb2}^2 \right) \|\Delta \psi\|_{\leb2} \|\nabla \psi\|_{\leb2} \\
    &\le  \eps \|\nabla \varphi\|_{\leb2}^2 + c_3 \left(\|\Delta \psi\|_{\leb2}^2 \|\nabla \psi\|_{\leb2}^2+1\right) \|\varphi\|_{\leb2}^2
  \end{align*}
  for all $(\varphi, \psi) \in X$,
  which implies \eqref{eq:gni_comb:lady} and \eqref{eq:gni_comb:comb} for $C \defs \max\{c_2^2, c_3\}$.
\end{proof}

As the second and third subproblem in \eqref{prob:gen} share some structural properties,
it appears sensible to study the quite general convection--diffusion equation
\begin{align}\label{prob:z}
  \begin{cases}
    z_t = d \Delta z - \nabla \cdot (\sigma(z) \nabla \psi_1) + \psi_2 & \text{in $\Omega \times (0, T)$}, \\
    \partial_\nu z = 0                                                 & \text{on $\Omega \times (0, T)$}, \\
    z(\cdot, 0) = z_0                                                  & \text{in $\Omega$}.
  \end{cases}
\end{align}
Provided that $d > 0$ and that the given functions $\sigma, \psi_1, \psi_2, z_0$ are bounded in appropriate spaces,
we obtain a~priori estimates for solutions of \eqref{prob:z} by employing a testing procedure and Ladyzhenskaya's trick (cf.\ \cite{LadyzenskajaSolutionLargeNonstationary1959}).
In less general settings, corresponding estimates have for instance been obtained in \cite[Lemma~3.2]{JinEtAlGlobalSolvabilityStability2023} by a similar method.
\begin{lem}\label{lm:gen_l2_bdd}
  Let $\Omega \subset \R^2$ be a smooth, bounded domain, $T \in (0,\infty)$, $M > 0$ and $d > 0$.
  Then there exists $C > 0$ such that for all $\sigma \in C^1([0, \infty))$, $\psi_1 \in C^{2,0}(\Ombar \times [0, T); \R^2)$, $\psi_2 \in C^0(\Ombar \times [0, T))$ and $z_0 \in \con0$
  with $|\sigma| \le \mathrm{id}$, $\partial_\nu \psi_1 = 0$ on $\partial \Omega$, $z_0 \ge 0$,
  \begin{align}\label{eq:gen_l2_bdd:cond}
    \sup_{t \in (0, T)} \intom |\nabla \psi_1(\cdot, t)|^2 \le M,
    \quad
    \intntom |\Delta \psi_1|^2 \le M
    \quad \text{and} \quad
    \intom z_0^2 \le M,
  \end{align}
  all nonnegative solutions $z \in C^0(\Ombar \times [0, T)) \cap C^{2, 1}(\Ombar \times (0, T))$ of \eqref{prob:z} satisfy
  \begin{align}\label{eq:gen_l2_bdd:statement}
    \sup_{t \in (0, T)} \intom z^2(\cdot, t) + \intntom |\nabla z|^2 + \intntom z (\psi_2)_- \le C \intntom z (\psi_2)_+ + C,
  \end{align}
  where $\xi_+ \defs \max\{\xi, 0\}$ and $\xi_- \defs \max\{-\xi, 0\}$ denote the positive and negative part of a real number $\xi$, respectively.
\end{lem}
\begin{proof}
  Testing \eqref{prob:z} with $z$ and applying Young's inequality yields 
  \begin{align*}
          \frac12 \ddt \intom z^2
    &=    - d \intom |\nabla z|^2 + \intom \sigma(z) \nabla z \cdot \nabla \psi_1 + \intom z \psi_2 \\
    &\le  - \frac{d}{2} \intom |\nabla z|^2 + \frac{1}{2d} \intom z^2 |\nabla \psi_1|^2 + \intom z \psi_2
    \qquad \text{in $(0, T)$}.
  \end{align*}
  Here, Lemma~\ref{lm:gni_comb} provides $c_1 > 0$ such that
  \begin{align*}
          \frac{1}{2d} \intom z^2 |\nabla \psi_1|^2
    &\le  \frac{d}{4} \intom |\nabla z|^2 + c_1 \left( \intom |\Delta \psi_1|^2 \intom |\nabla \psi_1|^2 + 1 \right) \intom z^2 \\
    &\le  \frac{d}{4} \intom |\nabla z|^2 + \frac{c_2}{2} \left( \intom |\Delta \psi_1|^2 + 1 \right) \intom z^2
    \qquad \text{in $(0, T)$},
  \end{align*}
  where $c_2 \defs 2\max\{1, M\} c_1$.
  Therefore, $y \colon [0, T) \to \R$, $t \mapsto \intom z^2(\cdot, t)$, fulfils
  \begin{align*}
    y' \le c_2 \left( \intom |\Delta \psi_1|^2 + 1 \right) y - \frac{d}{2} \intom |\nabla z|^2 + 2 \intom z \psi_2
    \qquad \text{in $(0, T)$},
  \end{align*}
  so that an ODE comparison argument and the variation-of-constants formula assert
  \begin{align*}
          \intom z^2(\cdot, t)
    &\le  \ure^{c_2 \int_0^t \left(\intom |\Delta \psi_1(x, s)|^2 \dx + 1 \right) \ds} \intom z_0^2 \\
    &\pe  + \int_0^t \ure^{c_2 \int_s^t \left( \intom |\Delta \psi_1(x, \tau)|^2 \dx + 1\right) \dtau} \left( - \frac{d}{2} \intom |\nabla z(x, s)|^2 \dx + 2 \intom (z \psi_2)(x, s) \dx \right) \ds \\
    &\le  \ure^{c_2 (M+T)} M
          - \frac{d}{2} \int_0^t \intom |\nabla z|^2 
          - 2 \int_0^t \intom z (\psi_2)_-
          + 2 \ure^{c_2(M+T)} \int_0^t \intom z (\psi_2)_+
  \end{align*}
  for all $t \in (0, T)$.
  An application of the monotone convergence theorem then yields \eqref{eq:gen_l2_bdd:statement} for $C \defs \max\{\frac{2}{d}, 1\} \max\{M, 2\} \ure^{c_2(M+T)} > 0$.
\end{proof}

\section{Global classical solutions for the system without prey-taxis}\label{sec:xi=0}
Throughout this section, we fix a smooth, bounded domain $\Omega \subset \R^2$, functions, parameters and initial data as in \eqref{eq:intro:def_f}--\eqref{eq:intro:init}.
Importantly, we also set $\xi = 0$, that is, we consider the setting without prey-taxis.
Then Lemma~\ref{lm:local_ex} asserts that there is a nonnegative, unique maximal classical solution $(u, v, w) $ of regularity \eqref{eq:local_ex:reg} of \eqref{system}.
We fix this solution as well as its maximal existence time $\tmax = \tmax(u_0, v_0, w_0)$.

The goal of this section is to prove Theorem~\ref{th:xi=0}, i.e.\ that this solution is global in time.
To that end, we now collect several a~priori estimates which will eventually show that \eqref{eq:local_ex:ext} does not hold,
which according to Lemma~\ref{lm:local_ex} can only happen if $\tmax = \infty$.

The first such bound going beyond Lemma~\ref{lm:u_v_w_l1} makes use of the structure of $f$ and $g$, the comparison principle,  and the assumption $\xi = 0$.
\begin{lem}\label{lm:u_v_linfty}
  There is $C >0$ such that
  \begin{align*}
    \|u\|_{L^\infty(\Omega \times (0, \tmax))} \le C
    \quad \text{and} \quad
    \|v\|_{L^\infty(\Omega \times (0, \tmax))} \le C.
  \end{align*}
\end{lem}
\begin{proof}
  Since $\ol u \defs \max\{\|u_0\|_{\leb\infty}, \frac{\lambda_1}{\mu_1}\}$ is a supersolution of the first equation in \eqref{system}, $u$ is bounded from above.
  Then $\ol v \defs \max\{\|v_0\|_{\leb\infty}, \frac{\lambda_2}{\mu_2} + \frac{b_1 \ol u}{\mu_2}\}$ is a supersolution of the second equation in \eqref{system}
  so that also $v$ is bounded from above.
  Finally, nonnegativity of $u$ and $v$ has already been asserted in Lemma~\ref{lm:local_ex}.
\end{proof}

Since the shape of $f$ and $g$ rather directly yield boundedness of $u$ and $v$ (as evidenced by Lemma~\ref{lm:u_v_linfty}), the situation of $\xi=0$ presents itself as similar to chemotaxis-consumption systems with one equation resembling 
\[
 w_t=Δw - ∇\cdot (w∇z) + \tilde{h}(w)
\]
for some bounded function $z$ and, in this case, essentially logistic source terms $\tilde{h}$. For this setting, global existence in two-dimensional domains is not surprising (cf.\ \cite{lanwin_consumptionsurvey}, \cite{lankeit_wang}), although some arguments relying on delicate energy-type arguments may not be transferable.

As $\xi = 0$, the $L^2$ space-time estimates provided by Lemma~\ref{lm:u_v_w_l1} rapidly imply a~priori estimates for certain spatial derivatives of $u$ and $v$.
\begin{lem}\label{lm:u_v_w12}
  Let $T \in (0, \infty) \cap (0, \tmax]$.
  Then there is $C > 0$ such that
  \begin{align*}
    \sup_{t \in (0, T)} \left( \intom |\nabla u(\cdot, t)|^2 + \intom |\nabla v(\cdot, t)|^2 \right) \le C
  \end{align*}
  and
  \begin{align*}
    \intntom |\Delta u|^2 + \intntom |\Delta v|^2 \le C.
  \end{align*}
\end{lem}
\begin{proof}
  Testing the first equation with $-\Delta u$ gives
  \begin{align*}
          \frac12 \ddt \intom |\nabla u|^2
    &=    - \intom |\Delta u|^2
          - \intom u (1 - u + v - w) \Delta u \\
    &\le  - \frac12 \intom |\Delta u|^2
          + \frac12 \|u\|_{L^\infty(\Omega \times (0, T))}^2 \intom (1 + u + v + w)^2 
    \qquad \text{in $(0, \tmax)$},
  \end{align*}
  so that due to Lemma~\ref{lm:u_v_w_l1} and Lemma~\ref{lm:u_v_linfty} the statement for the first solutions component follows upon an integration in time.
  The estimate for the second one follows analogously.
\end{proof}

The previous two lemmata make Lemma~\ref{lm:gen_l2_bdd} applicable and thus yield a uniform-in-time $L^2$ bound for $w$.
\begin{lem}\label{lm:w_l2}
  Let $T \in (0, \infty) \cap (0, \tmax]$.
  Then there is $C > 0$ such that
  \begin{align*}
    \sup_{t \in (0, T)} \intom w^2(\cdot, t) \le C.
  \end{align*}
\end{lem}
\begin{proof}
  We set $\psi_1 = \chi uv$ and $\psi_2 = h$, then Lemma~\ref{lm:u_v_linfty}, Lemma~\ref{lm:u_v_w12} and continuity of $w_0$ imply \eqref{eq:gen_l2_bdd:cond} for some $M > 0$,
  so that Lemma~\ref{lm:gen_l2_bdd} asserts
  \begin{align*}
    \sup_{t \in (0, T)} \intom w^2(\cdot, t) \le C \int_0^T \int_\Omega w^2 (\lambda_3 + b_2 u + b_3 v).
  \end{align*}
  The right-hand side herein is bounded by \eqref{eq:u_v_w_l1:l2l2} and Lemma~\ref{lm:u_v_linfty}.
\end{proof}

With these estimates at hand, showing $\tmax = \infty$ already comes down to employing a rather standard bootstrap procedure.
\begin{lem}\label{lm:w_linfty}
  The solution $(u, v, w)$ of \eqref{system} is global in time.
\end{lem}
\begin{proof}
  Suppose $\tmax \lt \infty$.
  According to \cite[Lemma~1.3]{winklerAggregationVsGlobal2010} (and because $\Omega$ is a two-dimensional domain), there are $c_1, c_2, c_3, c_4 \gt 0$ such that
  \begin{alignat*}{2}
    \|\ure^{t d_i \Delta} \varphi\|_{\leb p}
    &\le c_1 (1 + t^{- (\frac1q - \frac1p)}) \|\varphi\|_{\leb q}
    &&\qquad \text{for all $\varphi \in \leb q$}, \\
    \|\nabla \ure^{t d_i \Delta} \varphi\|_{\leb p}
    &\le c_2 (1 + t^{-\frac12 - (\frac1q - \frac1p)}) \|\varphi\|_{\leb q}
    &&\qquad \text{for all $\varphi \in \sob1q$}, \\
    \|\nabla \ure^{t d_i \Delta} \varphi\|_{\leb p}
    &\le c_3 \|\nabla \varphi\|_{\leb p}
    &&\qquad \text{for all $\varphi \in \sob1p$, provided $p \ge 2$}, \\
    \|\ure^{t d_i \Delta} \nabla \cdot \varphi\|_{\leb p}
    &\le c_4 (1 + t^{-\frac12 - (\frac1q - \frac1p)}) \|\varphi\|_{\leb q}
    &&\qquad \text{for all $\varphi \in L^q(\Omega; \R^2)$},
  \end{alignat*}
  for all $t \in (0, \infty)$, all $1 \le p \le q \le \infty$ and all $i \in \{1, 2, 3\}$.
  By \eqref{eq:intro:init}, there exists $\theta > 2$ with $u_0, v_0 \in \sob1\theta$.
  Thus, by the variations-of-constants formula,
  \begin{align*}
    &  \|\nabla u(\cdot, t)\|_{\leb\theta} 
    \le  \|\nabla \ure^{t d_1 \Delta} u_0\|_{\leb\theta}
          + \int_0^t \|∇\ure^{(t-s)d_1 \Delta} (u(\lambda_1 - \mu_1 u - a_1 v - a_2 w))(\cdot, s)\|_{\leb\theta} \ds \\
    &\le  c_3 \|u_0\|_{\sob1\theta}
          + c_2 \int_0^t (1 + (t-s)^{-\frac12 - (\frac12 - \frac1\theta)}) \|(u(\lambda_1 - \mu_1 u - a_1 v - a_2 w))(\cdot, s)\|_{\leb2} \ds \\
    &\le  c_3 \|u_0\|_{\sob1\theta}
          + c_2 \|u(\lambda_1 - \mu_1 u - a_1 v - a_2 w)\|_{L^\infty((0, \tmax); \leb2)} \int_0^{\tmax} (1 + s)^{-1+\frac1\theta} \ds
  \end{align*}
  for all $t \in (0, \tmax)$, which when combined with Lemma~\ref{lm:u_v_linfty} and Lemma~\ref{lm:w_l2} implies that there is $c_5 \gt 0$ such that
  \begin{align}\label{eq:w_linfty:u_est}
      \|u(\cdot, t)\|_{\sob1\theta} \le c_5
      \qquad \text{for all $t \in (0, \tmax)$.}
  \end{align}
  Likewise, we see that
  \begin{align}\label{eq:w_linfty:v_est}
      \|v(\cdot, t)\|_{\sob1\theta} \le c_6
      \qquad \text{for all $t \in (0, \tmax)$}
  \end{align}
  for some $c_6 \gt 0$.

  Fixing $q \in (2, \theta)$, we make again use of the variations-of-constants formula to obtain
  \begin{align*}
          \|w(\cdot, t)\|_{\leb\infty}
    &\le  \|\ure^{t d_3 \Delta} w_0\|_{\leb\infty}
          + \chi \int_0^t \|\ure^{(t-s)d_3\Delta} \nabla \cdot (w \nabla(uv))(\cdot, s)\|_{\leb\infty} \ds \\
    &\pe  + \int_0^t \|\ure^{(t-s)d_3\Delta} (w(\lambda_3 - \mu_3 w + a_2 u + a_2 v))(\cdot, s)\|_{\leb\infty} \ds \\
    &\le  \|w_0\|_{\leb\infty}
          + c_4 \chi \int_0^t (1 + (t-s)^{-\frac12 - \frac1q}) \|(w \nabla(uv))(\cdot, s)\|_{\leb q} \ds \\
    &\pe  + c_1 \int_0^t (1 + (t-s)^{-\frac34}) \|(w(\lambda_3 - \mu_3 w + a_2 u + a_2 v))(\cdot, s)\|_{\leb{\frac43}} \ds
  \end{align*}
  for all $t \in (0, \tmax)$.
  Here we make use of Hölder's inequality to obtain with $r \defs \frac{q\theta}{\theta-q} > 0$ that
  \begin{align*}
          \|(w \nabla(uv))(\cdot, s)\|_{\leb q}
    &\le  \|w(\cdot, s)\|_{\leb r} \|(\nabla(uv))(\cdot, s)\|_{\leb \theta} \\
    &\le  \|w(\cdot, s)\|_{\leb\infty}^{1-\frac1r} \|w(\cdot, s)\|_{\leb1}^\frac1r \|(\nabla(uv))(\cdot, s)\|_{\leb \theta}
  \intertext{and}
    &\pe  \|(w(\lambda_3 - \mu_3 w + a_2 + a_3 v))(\cdot, s)\|_{\leb{\frac43}} \\
    &\le  \|w(\cdot, s)\|_{\leb 4} \|(w(\lambda_3 - \mu_3 w + a_2 + a_3 v))(\cdot, s)\|_{\leb 2} \\
    &\le  \|w(\cdot, s)\|_{\leb \infty}^{1-\frac14} \|w(\cdot, s)\|_{\leb1}^\frac14 \|(w(\lambda_3 - \mu_3 w + a_2 + a_3 v))(\cdot, s)\|_{\leb 2}
  \end{align*}
  for all $s \in (0, \tmax)$.
  Therefore, there is $c_7 \gt 0$ such that
  \begin{align*}
    M \colon [0, \tmax) \ra [0, \infty), \quad t \mapsto 1 + \sup_{s \in (0, t)} \|w(\cdot, s)\|_{\leb\infty}
  \end{align*}
  fulfils
  \begin{align*}
        M(t)
    \le c_7 + c_7 M^{1-\frac1{\max\{4, r\}}}(t)
    \qquad \text{for all $t \in (0, \tmax)$}
  \end{align*}
  and hence
  \begin{align*}
        M^{\frac1{\max\{4, r\}}}(t)
    \le 2 c_7
    \qquad \text{for all $t \in (0, \tmax)$}.
  \end{align*}
  In particular, there is $c_8 \gt 0$ such that
  \begin{align}\label{eq:w_linfty_w_est}
      \|w(\cdot, t)\|_{\leb\infty} \le c_8
      \qquad \text{for all $t \in (0, \tmax)$}.
  \end{align}

  In combination, \eqref{eq:w_linfty:u_est}, \eqref{eq:w_linfty:v_est}, \eqref{eq:w_linfty_w_est} and the extensibility criterion in Lemma~\ref{lm:local_ex}
  show that our assumption $\tmax \lt \infty$ must be false.
\end{proof}

\begin{proof}[Proof of Theorem~\ref{th:xi=0}]
  All claims have been established in Lemma~\ref{lm:w_linfty}.
\end{proof}

\section{Global generalized solutions for the system with prey-taxis}\label{sec:xi>0}
\subsection{Solution concept}
In this section, we will construct global generalized solution of \eqref{system} with $\xi > 0$.
We begin by introducing our solution concept.
\begin{definition}\label{def:sol_concept}
  Let $\Omega \subset \R^2$ be a smooth, bounded domain, assume \eqref{eq:intro:def_f}--\eqref{eq:intro:def_h} and let $u_0, v_0, w_0 \in \leb1$ be nonnegative.
  A triple $(u, v, w) \in L_{\loc}^2(\Ombarinf)$ of nonnegative functions
  with
  \begin{align*}
    \nabla u,\; \nabla v,\; \mathds 1_{\{w \le k\}} \nabla w \in L_{\loc}^2(\Ombarinf)
  \end{align*}
  for all $k \in \N$ (where we denote the characteristic function of a set $A$ by $\mathds 1_A$) 
  is called a \emph{global generalized solution} of \eqref{system} if
  \begin{itemize}
    \item
      $u$ and $v$ are weak solutions of the respective subproblems in \eqref{system}, that is, 
      \begin{align}\label{eq:sol_concept:u_weak}
          - \intninfom u \varphi_t
          - \intom u_0 \varphi(\cdot, 0) 
        = - d_1 \int_0^\infty\! \intom \nabla u \cdot \nabla \varphi
          + \int_0^\infty\!\intom f(u, v, w) \varphi
      \end{align}
      and
      \begin{align}\label{eq:sol_concept:v_weak}
          - \int_0^\infty\!\!\!\intom v \varphi_t
          - \intom v_0 \varphi(\cdot, 0) 
        = - d_2 \int_0^\infty\!\!\!\intom \nabla v \cdot \nabla \varphi
          + \xi \int_0^\infty\!\!\!\intom v \nabla u \cdot \nabla \varphi
          + \int_0^\infty\!\!\!\intom g(u, v, w) \varphi
      \end{align}
      hold for all $\varphi \in C_c^\infty(\Ombarinf)$,

    \item
      for all $\phi \in C^\infty([0, \infty)^2)$ with $D \phi \in C_c^\infty([0, \infty)^2)$ and $\phi_{ww} \le 0$ in $[0, \infty)^2$,
      $\phi(v, w)$ is a weak $\phi$-supersolution of the corresponding subproblem in \eqref{system} in the sense that
      \begin{align}\label{eq:sol_concept:vw_phi_supersol}
        &\pe  - \intninfom \phi(v, w) \varphi_t
              - \intom \phi(v_0, w_0) \varphi(\cdot, 0) \notag \\
        &\ge  - \intninfom \big(d_2 \nabla v - \xi v \nabla u\big) \cdot \big(\phi_{vv}(v, w) \nabla v \varphi + \phi_{vw}(v, w) \nabla w \varphi + \phi_{v}(v, w) \nabla \varphi\big) \notag \\
        &\pe  - \intninfom \big(d_3 \nabla w - \chi w \nabla (uv)\big) \cdot \big(\phi_{vw}(v, w) \nabla v \varphi + \phi_{ww}(v, w) \nabla w \varphi + \phi_{w}(v, w) \nabla \varphi\big) \notag \\
        &\pe  + \intninfom g(u, v, w) \phi_v(v, w) \varphi
              + \intninfom h(u, v, w) \phi_w(v, w) \varphi
      \end{align}
      holds for all nonnegative $\varphi \in C_c^\infty(\Ombarinf)$ and

    \item
      there is a null set $N \subset (0, \infty)$ such that
      \begin{align}\label{eq:sol_def:mass_ineq}
        \intom w(\cdot, t) \le \intom w_0 + \int_0^t \intom h(u, v, w)
        \qquad \text{for all $t \in (0, \infty) \setminus N$.}
      \end{align}
  \end{itemize}  
\end{definition}

This concept is consistent with the notion of classical solutions.
\begin{lem}\label{lm:sol_concept_consistent}
  Let $\Omega \subset \R^2$ be a smooth, bounded domain, assume \eqref{eq:intro:def_f}--\eqref{eq:intro:def_h}, let $u_0, v_0, w_0 \in \leb1$ be nonnegative
  and let $(u, v, w)$ be a global generalized solution of \eqref{system} in the sense of Definition~\ref{def:sol_concept},
  which additionally fulfils $u, v, w \in C^{2, 1}(\Ombar \times (0, \infty)) \cap C^0(\Ombarinf)$.
  Then $(u, v, w)$ is also a classical solution of \eqref{system}.
\end{lem}
\begin{proof}
  This can be shown similarly as in \cite[Lemma~5.3]{FuestStrongConvergenceWeighted2023} and \cite[Lemma~2.1]{WinklerLargedataGlobalGeneralized2015}.
\end{proof}

\begin{rem}
 For the integral terms in Definition~\ref{def:sol_concept} to be well-defined, slightly less regularity would suffice; for example, one could require $\nabla v \in L^1$ and $\mathds 1_{\{v \le k\}} \nabla v\in L^2$ instead of $\nabla v \in L^2$. While we rely on the strong $L^2$ \textit{convergence} of $\mathds 1_{\{\ve \le k\}} \nabla \ve$ (cf.\ Lemma~\ref{lm:strong_grad_conv}), it is even more easily obtained that $\nabla v$ (without the additional cutoff) belongs to $L_{\loc}^2(\Ombarinf)$ (see \eqref{eq:eps_sea_0:nabla_v_l2}). On the other hand, posing an integrability condition on $\mathds 1_{\{w \le k\}} \nabla w$ and not on $\nabla w$ is crucial, as this allows to conclude the needed boundedness from an estimate of $\nabla \ln(w+1)$ instead of $\nabla w$ itself. (See Lemma~\ref{lm:nabla_we_l2} and \eqref{eq:eps_sea_0:nabla_ln_w_l2}.)
\end{rem}

\subsection{A priori estimates for solutions to an approximate problem}\label{sec:xi>0:apriori}
Henceforth, we fix a smooth, bounded domain $\Omega \subset \R^2$, the functions $f, g, h$ defined in \eqref{eq:intro:def_f}--\eqref{eq:intro:def_h}, parameters as in \eqref{eq:intro:params}, where $\xi > 0$,
and nonnegative initial data $u_0 \in \sob12$, $v_0 \in \leb2$ and $w_0 \in \leb1$.
For each $\eps \in (0, 1)$, this allows us to also fix nonnegative $\une, \vne, \wne \in \con\infty$ with $∂_{ν}\une=∂_{ν}\vne=∂_{ν}\wne=0$ on $∂\Om$ and such that
\begin{align}\label{eq:conv_init}
  (\une, \vne, \wne) \to (u_0, v_0, w_0) \qquad \text{in $\sob12 \times \leb2 \times \leb1$ as $\eps \sea 0$}.
\end{align}
Moreover, we fix $\sigma \in C^\infty(\R; [0, 1])$ with $\sigma(s) = 1$ for $s \le 0$ and $\sigma(s) = 0$ for $s \ge 1$
and set $\sigmae(s) = s \sigma(\eps s - 1)$ for $s \ge 0$ and $\eps \in (0, 1)$.
Then
\begin{align}\label{eq:sigmae_1}
  \sigmae(s) \begin{cases}
    = s, & s \le \frac1\eps, \\
    \in [0, s], & \frac1\eps \le s \le \frac2\eps, \\
    = 0, & s \ge \frac2\eps,
  \end{cases}
\end{align}
and
\begin{align}\label{eq:sigmae_2}
  |\sigmae'(s)| = |\sigma(\eps s - 1) + \eps s \sigma'(\eps s - 1)| \le 1 + 2 \|\sigma'\|_{C^0([0,1])}
\end{align}
for $s \ge 0$ and $\eps \in (0, 1)$.

Again for each $\eps \in (0, 1)$,
Lemma~\ref{lm:local_ex} then asserts that there exist $\tmaxe \in (0, \infty]$
and a nonnegative maximal classical solution $(\ue, \ve, \we) \in \big(C^{2, 1}(\Ombar \times (0, \tmaxe)) \cap C^0(\Ombar \times [0, \tmaxe))\big)^3$ of
\begin{align}\label{prob:approx}
  \begin{cases}
    \uet = d_1 \Delta \ue                                                     + f(\ue, \ve, \we) & \text{in $\Omega \times (0, \tmaxe)$}, \\
    \vet = d_2 \Delta \ve - \xi \nabla \cdot (\sigmae(\ve) \nabla \ue)        + g(\ue, \ve, \we) & \text{in $\Omega \times (0, \tmaxe)$}, \\
    \wet = d_3 \Delta \we - \chi \nabla \cdot (\sigmae(\we) \nabla (\ue \ve)) + h(\ue, \ve, \we) & \text{in $\Omega \times (0, \tmaxe)$}, \\
    \partial_\nu \ue = \partial_\nu \ve = \partial_\nu \we = 0                                   & \text{on $\partial \Omega \times (0, \tmaxe)$}, \\
    \ue(\cdot, 0) = \une, \ve(\cdot, 0) = \vne, \we(\cdot, 0) = \wne                             & \text{in $\Omega$}.
  \end{cases}  
\end{align}

As in Section~\ref{sec:xi=0}, the lack of a cross-diffusive term together with the structure of $f$ implies several a~priori estimates for the first solution component.
\begin{lem}\label{lm:ue_bdds}
  For all finite $T \in (0, \tmaxe]$, there exists $C > 0$ such that 
  \begin{align*}
    \sup_{t \in (0, T)} \|\ue(\cdot, t)\|_{\leb\infty} \le C, \quad
    \sup_{t \in (0, T)} \intom |\nabla \ue(\cdot, t)|^2 \le C \quad \text{and} \quad
    \intntom |\Delta \ue|^2 \le C
  \end{align*}
  for all $\eps \in (0, 1)$.
\end{lem}
\begin{proof}
  This can be shown as in Lemma~\ref{lm:u_v_linfty} and Lemma~\ref{lm:u_v_w12}.
\end{proof}

When $\xi$ is positive, however, bounds for the second solution component are not as easily obtained as in Section~\ref{sec:xi=0};
both the comparison principle used in Lemma~\ref{lm:u_v_linfty} and the testing procedure employed in Lemma~\ref{lm:u_v_w12} are no longer applicable.
Fortunately, we can at least make use of Lemma~\ref{lm:gen_l2_bdd} to obtain the following
\begin{lem}\label{lm:ve_bdds}
  For all finite $T \in (0, \tmaxe]$, there exists $C > 0$ such that
  \begin{align*}
    \sup_{t \in (0, T)} \intom \ve^2(\cdot, t) \le C, \quad
    \intntom \ve^3 \le C \quad \text{and} \quad
    \intntom |\nabla \ve|^2 \le C
  \end{align*}
  for all $\eps \in (0, 1)$.
\end{lem}
\begin{proof}
  As Lemma~\ref{lm:ue_bdds} and \eqref{eq:conv_init} provide $\eps$-independent bounds for the quantities in \eqref{eq:gen_l2_bdd:cond} (with $\psi_1 = \ue$ and $z_0 = \vne$),
  we can apply Lemma~\ref{lm:gen_l2_bdd} to obtain $c_1 > 0$ such that
  \begin{align*}
    \sup_{t \in (0, T)} \intom \ve^2(\cdot, t) + \intntom |\nabla \ve|^2 + \intntom \ve^2 (\mu_1 \ve + a_3 \we)
    \le c_1 \intntom \ve^2 (\lambda + b_1 \ue)
  \end{align*}
  for all $\eps \in (0, 1)$. As the right-hand side is bounded by Lemma~\ref{lm:u_v_w_l1} and Lemma~\ref{lm:ue_bdds}, this yields the desired estimates.
\end{proof}

Before collecting further $\eps$-independent a~priori estimates, we briefly state that the approximate solutions exist globally.
To that end, we already make use of Lemma~\ref{lm:ue_bdds}.
\begin{lem}\label{lm:globale_ex_eps}
  For all $\eps \in (0, 1)$, we have $\tmaxe=\infty$; that is, the solution $(\ue, \ve, \we)$ of \eqref{prob:approx} is global in time. 
\end{lem}
\begin{proof}
  Suppose there is $\eps \in (0, 1)$ such that $\tmaxe$ is finite.
  According to Lemma~\ref{lm:ue_bdds}, there is $c_1 > 0$ with $\ue \le c_1$ in $\Omega \times (0, \tmaxe)$.
  As $\sigmae(s) = 0$ for $s \ge \frac{2}{\eps}$ by \eqref{eq:sigmae_1},
  the constant function $\ol{\ve} \equiv \max \{\|\vne\|_{\leb \infty}, \frac{2}{\eps}, \frac{\lambda_2 + b_1 c_1}{\mu_2}\}$ is a supersolution to the second equation in \eqref{prob:approx}.
  Similarly, $\ol{\we} \equiv \max \{\|\wne\|_{\leb \infty}, \frac{2}{\eps}, \frac{\lambda_3 + b_2 c_1 + b_3 \ol \ve}{\mu_3}\}$ is a supersolution to the third equation in \eqref{prob:approx},
  so that all solution components and thus also the zeroth order terms in \eqref{prob:approx} are bounded in $L^\infty(\Omega \times (0, \tmaxe))$.

  Next, maximal Sobolev regularity (cf.\ \cite[Theorem~2.3]{GigaSohrAbstractEstimatesCauchy1991}) asserts that both $\uet$ and $\Delta \ue$ belong to $X \defs L^{12}(\Omega \times (0, \tmaxe))$,
  whence $\ue \in Y \defs C^{\frac53, \frac56}(\Ombar \times [0, \tmaxe])$ by \cite[Lemma~II.3.3]{LadyzenskajaEtAlLinearQuasilinearEquations1988}.
  Since then $\ve \Delta \ue \in X$ and $\nabla \ue \in L^\infty((0, \tmaxe); L^{12}(\Omega))$,
  we may apply a consequence of maximal Sobolev regularity theory, namely \cite[Lemma~2.13]{FuestEtAlGlobalExistenceClassical2023},
  to also obtain $\vet, \Delta \ve \in X$. Again due to \cite[Lemma~II.3.3]{LadyzenskajaEtAlLinearQuasilinearEquations1988}, this gives $\ve \in Y$,
  so that, in conclusion, $(\ue, \ve, \we) \in L^\infty((0, \tmaxe); \con1 \times \con1 \times \con0)$.
  However, this contradicts \eqref{eq:local_ex:ext}, hence $\tmaxe$ cannot be finite.
\end{proof}

As a consequence of Lemma~\ref{lm:ue_bdds} and Lemma~\ref{lm:ve_bdds},
we see that the second equation in \eqref{prob:approx} can be written as a heat equation with a force term uniformly bounded in $L_{\loc}^{6/5}(\Ombarinf)$.
\begin{lem}\label{lm:nabla_ve_65}
  For all $T > 0$, there exists $C > 0$ such that
  \begin{align}\label{eq:nabla_ve_65:statement}
    \intntom |-\xi \nabla \cdot (\sigmae(\ve) \nabla \ue) + g(\ue, \ve, \we)|^\frac65 \le C
    \qquad \text{for all $\eps \in (0, 1)$}.
  \end{align}
\end{lem}
\begin{proof}
  By Hölder's and Young's inequalities as well as \eqref{eq:sigmae_2}, we have
  \begin{align*}
    &\pe  \|-\xi \nabla \cdot (\sigmae(\ve) \nabla \ue) + g(\ue, \ve, \we)\|_{\leb[Q_T]{6/5}} \\
    &\le  \xi \|\sigmae'(\ve) \nabla \ue \cdot \nabla \ve\|_{\leb[Q_T]{6/5}}
          + \xi \|\sigmae(\ve) \Delta \ue\|_{\leb[Q_T]{6/5}}
          + c_1 \|\ve(1 + \ue + \ve + \we)\|_{\leb[Q_T]{6/5}} \\
    &\le  c_2 \left( \|\nabla \ue\|_{\leb[Q_T]{3}}^3 +  \|\nabla \ve\|_{\leb[Q_T]{2}}^2 \right)
          + \xi \left( \|\ve\|_{\leb[Q_T]{3}}^3 + \|\Delta \ue\|_{\leb[Q_T]{2}}^2 \right)\\
    &\pe  + c_1 \left( \|\ve\|_{\leb[Q_T]{3}}^3 + \|1 + \ue + \ve + \we\|_{\leb[Q_T]{2}}^2 \right)
  \end{align*}
  for all $\eps \in (0, 1)$ and some $c_1, c_2 > 0$, where $Q_T \defs \Omega \times (0, T)$.
  As moreover
  \begin{align*}
        \|\nabla \ue\|_{\leb[Q_T]{3}}^3
    \le \|\nabla \ue\|_{\leb[Q_T]{4}}^4 + c_3
    \le c_4 \|\ue\|_{L^\infty((0, T); \leb2)}^2 \|\Delta \ue\|_{\leb[Q_T]{2}}^2 + c_3
  \end{align*}
  for all $\eps \in (0, 1)$ and some $c_3, c_4 > 0$ by Young's inequality and \eqref{eq:gni_comb:lady},
  we conclude \eqref{eq:nabla_ve_65:statement} upon applying Lemma~\ref{lm:ue_bdds}, Lemma~\ref{lm:ve_bdds} and Lemma~\ref{lm:u_v_w_l1}.
\end{proof}

Although the second and third equation in \eqref{prob:approx} are structurally similar in some sense,
we cannot repeat the reasoning in Lemma~\ref{lm:ve_bdds} (nor the one in Lemma~\ref{lm:w_l2}) --
our known $\eps$-independent bounds for $uv$ are much worse than for $u$.
Instead, we shall employ a testing procedure in order to obtain a~priori estimates for the gradient of $\we$
which, while weaker than those for the other solution components, will still turn out to be sufficient for our purposes.
\begin{lem}\label{lm:nabla_we_l2}
  For all $T \in (0, \infty)$, there exists $C > 0$ such that 
  \begin{align*}
    \intntom |\nabla (\ue \ve)|^2 \le C \quad \text{and} \quad \intntom \frac{|\nabla \we|^2}{(\we+1)^2} \le C
    \qquad \text{for all $\eps \in (0, 1)$}.
  \end{align*}
\end{lem}
\begin{proof}
  By testing the third equation in \eqref{prob:approx} with $-\frac{1}{\we+1}$, we obtain
  \begin{align}\label{eq:nabla_we_l2:test}
          - \ddt \intom \ln(\we+1)
    &=    - d_3 \intom \frac{|\nabla \we|^2}{(\we+1)^2}
          + \chi \intom \frac{\sigmae(\we)}{(\we+1)^2} \nabla \we \cdot \nabla (\ue \ve)
          - \intom \frac{h(\ue, \ve, \we)}{\we+1} \notag \\
    &\le  - \frac{d_3}{2} \intom \frac{|\nabla \we|^2}{(\we+1)^2}
          + \frac{\chi^2}{2d_3} \intom |\nabla (\ue \ve)|^2
          + \mu_3 \intom \frac{\we^2}{\we+1}
  \end{align}
  in $(0, T)$ for all $\eps \in (0, 1)$.
  As \eqref{eq:gni_comb:comb}, Lemma~\ref{lm:ue_bdds} and Lemma~\ref{lm:ve_bdds} provide $c_1, c_2 > 0$ such that
  \begin{align*}
          \intom |\nabla (\ue \ve)|^2
    &\le  2 \intom \ue^2 |\nabla \ve|^2 + 2 \intom \ve^2 |\nabla \ue|^2 \\
    &\le  2 \|\ue\|_{L^\infty(\Omega \times (0, T))}^2 \intom |\nabla \ve|^2
          + \intom |\nabla \ve|^2 + c_1 \left( \intom |\Delta \ue|^2 \intom |\nabla \ue|^2 + 1 \right) \intom \ve^2 \\
    &\le  c_2 \intom |\Delta \ue|^2 + c_2
    \qquad \text{in $(0, T)$ for all $\eps \in (0, 1)$,}
  \end{align*}
  integrating \eqref{eq:nabla_we_l2:test} yields
  \begin{align*}
    &\pe  \intom \ln(\wne+1) + \frac{d_3}{2} \int_0^t \intom \frac{|\nabla \we|^2}{\we+1} + \int_0^t \intom |\nabla (\ue \ve)|^2 \\
    &\le  \intom \ln(\we(\cdot, t)) + c_3 \int_0^t \intom |\Delta \ue|^2 + c_3 t + \mu_3 \int_0^t \intom \we \\
    &\le  \intom \we(\cdot, t) + c_3 \int_0^T \intom |\Delta \ue|^2 + c_3 T + \mu_3 \int_0^T \intom \we
  \end{align*}
  for all $t \in (0, T)$ and $\eps \in (0, 1)$, where $c_3 \defs \frac{\chi^2 c_2}{2d_3}+c_2$.
  The statement then follows by Lemma~\ref{lm:u_v_w_l1}, Lemma~\ref{lm:ue_bdds} and the monotone convergence theorem. 
\end{proof}

In order to prepare applications of the Aubin--Lions lemma,
we next collect estimates for the time derivatives which rapidly follow from the bounds obtained above.
\begin{lem}\label{lm:uet_vet_wet}
  For all $T \in (0, \infty)$, there is $C > 0$ such that   
  \begin{align}\label{eq:uet_vet_wet:bdd_uv}
    \|\uet\|_{L^2(\Omega \times (0, T))} \le C, \quad
    \|\vet\|_{L^{6/5}((0, T); (\sob12)^\star)} \le C
  \end{align}
  and
  \begin{align}\label{eq:uet_vet_wet:bdd_w}
    \|(\ln(\we+1))_t\|_{L^1((0, T); (\sob22)^\star)} \le C
  \end{align}
  for all $\eps \in (0, 1)$.
\end{lem}
\begin{proof}
  As to \eqref{eq:uet_vet_wet:bdd_w}, we test the third equation in \eqref{prob:approx} with $\varphi \in C^\infty(\Ombar)$ to obtain that with some $c_1 > 0$
  \begin{align*}
    &\pe  \left| \intom (\ln(\we+1))_t \varphi \right|
     =    \left| \intom \wet \frac{\varphi}{\we+1} \right| \\
    &\le  \left| \intom (d_1 \nabla \we - \chi \sigmae(\we) \nabla (\ue \ve)) \cdot \nabla \frac{\varphi}{\we+1} \right|
          + \left| \intom \frac{h(\ue, \ve, \we) \varphi}{\we+1} \right| \\
    &\le  d_1 \intom \frac{|\nabla \we|}{\we+1} |\nabla \varphi|
          + d_1 \intom \frac{|\nabla \we|^2}{(\we+1)^2} |\varphi| \\
    &\pe  + \chi \intom \frac{\we}{\we+1} |\nabla(\ue \ve)| |\nabla \varphi|
          + \chi \intom \frac{\we|\nabla \we|}{(\we+1)^2} |\nabla(\ue \ve)|  |\varphi|
          + \intom |h(\ue, \ve, \we)| |\varphi| \\
    &\le  c_1 \left( \intom \frac{|\nabla \we|^2}{(\we+1)^2} + \intom |\nabla(\ue \ve)|^2 + \intom |h(\ue, \ve, \we)| + 1 \right)
          \left( \|\varphi\|_{\leb\infty} + \|\nabla \varphi\|_{\leb2} \right)
  \end{align*}
  in $(0, T)$ for all $\eps \in (0, 1)$.
  Combined with the embedding $\sob22 \embed \leb\infty \cap \sob12$
  this implies that there exists $c_2 > 0$ with
  \begin{align}\label{eq:uet_vet_wet:ln_w_space}
    \|(\ln(\we+1))_t\|_{(\sob22)^\star}
    \le c_2 \left( \intom \frac{|\nabla \we|^2}{(\we+1)^2} + \intom |\nabla(\ue \ve)|^2 + \intom |h(\ue, \ve, \we)| +1 \right)
  \end{align}
  in $(0, T)$ for all $\eps \in (0, 1)$.
  Since Lemma~\ref{lm:nabla_we_l2} and Lemma~\ref{lm:u_v_w_l1} provide $\eps$-independent estimates for the $L^1((0, T))$ norm of the right-hand side of \eqref{eq:uet_vet_wet:ln_w_space}, we arrive at \eqref{eq:uet_vet_wet:bdd_w}.

  Finally, the second half of \eqref{eq:uet_vet_wet:bdd_uv} follows by a similar (but simpler) testing procedure from   Lemma~\ref{lm:ve_bdds}, Lemma~\ref{lm:nabla_ve_65}, 
  and the embedding $\sob12 \embed \leb6$, whereas the first part of \eqref{eq:uet_vet_wet:bdd_uv} is an immediate consequence of Lemma~\ref{lm:ue_bdds} and \eqref{eq:u_v_w_l1:l2l2}.
\end{proof}

\subsection{The limit process \texorpdfstring{$\eps \sea 0$}{epsilon to 0}: Obtaining solution candidates}\label{sec:xi>0:eps_sea_0}
With the a~priori estimates obtained in the previous subsection at hand, we are now able to take the limit of $(\ue, \ve, \we)$ (along some null sequence) in appropriate spaces.
\begin{lem}\label{lm:eps_sea_0}
  There exist nonnegative functions $u, v, w \colon \Omega \times (0, \infty) \to [0, \infty)$ and a null sequence $(\eps_j)_{j \in \N} \subset (0, 1)$ such that
  \begin{alignat}{2}
    \ue &\ra u
    &&\qquad \text{a.e.\ and in $L_{\loc}^2(\Ombarinf)$}, \label{eq:eps_sea_0:u_l2}\\
    \Delta \ue &\rh \Delta u
    &&\qquad \text{in $L_{\loc}^2(\Ombarinf)$}, \label{eq:eps_sea_0:delta_u_l2}\\
    \ve &\ra v
    &&\qquad \text{a.e.\ and in $L_{\loc}^2(\Ombarinf)$}, \label{eq:eps_sea_0:v_l2}\\
    \nabla \ve &\rh \nabla v
    &&\qquad \text{in $L_{\loc}^2(\Ombarinf)$}, \label{eq:eps_sea_0:nabla_v_l2}\\
    \we &\ra w
    &&\qquad \text{a.e.\ and in $L_{\loc}^1(\Ombarinf)$}, \label{eq:eps_sea_0:w_l1}\\
    \we &\rh w
    &&\qquad \text{in $L_{\loc}^2(\Ombarinf)$}, \label{eq:eps_sea_0:w_l2}\\
    \nabla \ln(\we+1) &\rh \nabla \ln(w+1)
    &&\qquad \text{in $L_{\loc}^2(\Ombarinf)$} \label{eq:eps_sea_0:nabla_ln_w_l2}
  \end{alignat}
  as $\eps = \eps_j \sea 0$.
\end{lem}
\begin{proof}
  According to \eqref{eq:u_v_w_l1:l1}, Lemma~\ref{lm:ue_bdds}, Lemma~\ref{lm:ve_bdds}, and Lemma~\ref{lm:nabla_we_l2}, 
  the families $(\ue)_{\eps \in (0, 1)}$, $(\ve)_{\eps \in (0, 1)}$ and $(\ln (\we+1))_{\eps \in (0, 1)}$ are bounded in $L_{\loc}^2([0, \infty); \sob12)$, 
  while by Lemma~\ref{lm:uet_vet_wet} their time derivatives are bounded in $L_{\loc}^2(\Ombar \times [0, \infty))$,  $L_{\loc}^{6/5}([0, \infty); (\sob12)^\star)$ and $L_{\loc}^1([0, \infty); (\sob22)^\star)$, respectively.
  Thus, the Aubin--Lions lemma and a diagonalization argument show that there are $u, v, z \in L_{\loc}^1(\Ombarinf)$ such that $(\uej, \vej, \ln(\wej+1)) \to (u, v, z)$ in $L_{\loc}^1(\Ombarinf)$
  along some null sequence $(\eps_j)_{j \in \N} \subset (0, 1)$.
  Upon switching to a subsequence and setting $w \defs \ure^{z}-1$, we may further assume that $((\uej, \vej, \wej))_{j \in \N}$ converges a.e.\ to $(u, v, w)$.
  Then the nonnegativity of the approximate solutions asserted in Lemma~\ref{lm:local_ex} is transferred to the limit functions,
  and \eqref{eq:eps_sea_0:u_l2}, \eqref{eq:eps_sea_0:v_l2} as well as \eqref{eq:eps_sea_0:w_l1} follow from Lemma~\ref{lm:ue_bdds}, Lemma~\ref{lm:ve_bdds}, \eqref{eq:u_v_w_l1:l2l2} and Vitali's theorem.

  Moreover, $(\Delta \ue)_{\eps \in (0, 1)}$, $(\nabla \ve)_{\eps \in (0, 1)}$, $(\we)_{\eps \in (0, 1)}$ and $(\nabla \ln(\we+1))_{\eps \in (0, 1)}$ are bounded in $L_{\loc}^2(\Ombarinf)$
  by Lemma~\ref{lm:ue_bdds}, Lemma~\ref{lm:ve_bdds}, \eqref{eq:u_v_w_l1:l2l2} and Lemma~\ref{lm:nabla_we_l2},
  so that they converge weakly along some further subsequence in $L_{\loc}^2(\Ombarinf)$.
  As pointwise a.e.\ convergence (of the functions themselves) is already established, we can identify their limits with $\Delta u$, $\nabla v$, $w$ and $\nabla \ln(w+1)$, respectively,
  and conclude \eqref{eq:eps_sea_0:delta_u_l2}, \eqref{eq:eps_sea_0:nabla_v_l2}, \eqref{eq:eps_sea_0:w_l2} and \eqref{eq:eps_sea_0:nabla_ln_w_l2}.
\end{proof}

In order to show that the triple $(u, v, w)$ obtained in Lemma~\ref{lm:eps_sea_0} is actually a generalized solution of \eqref{system} in the sense of Definition~\ref{def:sol_concept},
we in particular need to show that $\phi(u, w)$ is a weak $\phi$-supersolution for certain choices of $\phi$; that is, that \eqref{eq:sol_concept:vw_phi_supersol} holds.
However, the latter contains \emph{quadratic} expressions of (weighted) gradients of all solution components,
meaning that the weak convergence of gradient terms asserted by Lemma~\ref{lm:eps_sea_0} is yet insufficient.

Thus, we are interested in strong convergence of (weighted) gradients of $\ue$ and $\ve$.
(The term in \eqref{eq:sol_concept:vw_phi_supersol} which is quadratic in $\nabla w$ has a favorable sign and can be treated by making use of the weak lower semicontinuity of the norm.)
According to \cite{FuestStrongConvergenceWeighted2023}, this follows if the respective equations can be written as a heat equation with a force term converging in sufficiently strong topologies.
Fortunately, the latter is contained in the above analysis and thus we obtain
\begin{lem}\label{lm:strong_grad_conv}
  Let $(u, v, w)$ and $(\eps_j)_{j \in \N} \subset (0, 1)$ be as given by Lemma~\ref{lm:eps_sea_0}.
  Then there exists a subsequence of $(\eps_j)_{j \in \N}$, which we do not relabel, such that
  \begin{alignat}{2}
    \nabla \ue &\ra \nabla u
    &&\qquad \text{in $L_{\loc}^2(\Ombarinf)$}, \label{eq:strong_grad_conv:u}\\
    \mathds 1_{\{\ve \le k\}} \nabla \ve &\ra \mathds 1_{\{v \le k\}} \nabla v
    &&\qquad \text{in $L_{\loc}^2(\Ombarinf)$ for all $k \in \N$} \label{eq:strong_grad_conv:v}
  \end{alignat}
  as $\eps = \eps_j \sea 0$.
\end{lem}
\begin{proof}
  As Lemma~\ref{lm:ue_bdds} shows that $(\ue)_{\eps \in (0, 1)}$ is bounded in $L_{\loc}^\infty(\Ombarinf)$,
  \eqref{eq:eps_sea_0:u_l2}, \eqref{eq:eps_sea_0:v_l2}, \eqref{eq:eps_sea_0:w_l1} and a variant of Lebesgue's theorem assert
  \begin{align*}
    \ue f(\ue, \ve, \we) = \ue^2 \cdot (\lambda_1 - \mu_1 \ue - a_1 \ve - a_2 \we) \to u f(u, v, w)
  \end{align*}
  in $L_{\loc}^1(\Ombarinf)$ as $\eps = \eps_j \sea 0$.
  As also $\une \to u_0$ in $\leb2$ as $\eps \sea 0$ by \eqref{eq:conv_init}, \eqref{eq:strong_grad_conv:u}
  follows upon an application of \cite[(Theorem~1.1 and) Theorem~1.3]{FuestStrongConvergenceWeighted2023} (with the choice $\psi(s)=\frac{s^2}{2}$ there, and possibly after switching to a subsequence).
(Alternatively, the proof of \eqref{eq:strong_grad_conv:u} can be based on \eqref{eq:eps_sea_0:u_l2} and \eqref{eq:eps_sea_0:delta_u_l2}.)
  
  By Lemma~\ref{lm:nabla_ve_65}, 
  there exist $z \in L_{\loc}^{6/5}(\Ombarinf)$ and a subsequence of $(\eps_j)_{j \in \N}$, not relabeled, such that
  \begin{align*}
     -\xi \nabla \cdot (\sigmae(\ve) \nabla \ue) + g(\ue, \ve, \we) \rh z
    \qquad \text{in $L_{\loc}^{6/5}(\Ombarinf)$ as $\eps = \eps_j \sea 0$.}
  \end{align*}
  In combination with \eqref{eq:conv_init}, this renders \cite[Theorem~1.1]{FuestStrongConvergenceWeighted2023} applicable,
  which allows us to extract a final subsequence of $(\eps_j)_{j \in \N}$ such that
  \begin{align*}
    \ve \ra \tilde v
    \quad \text{a.e.}
    \qquad \text{and} \qquad
    \mathds 1_{\{\ve \le k\}} \nabla \ve &\ra \mathds 1_{\{\tilde v \le k\}} \nabla \tilde v
    \quad \text{in $L_{\loc}^2(\Ombarinf)$ for all $k \in \N$} 
  \end{align*}
  as $\eps = \eps_j \sea 0$ for some function $\tilde v \colon \Ombarinf \to \R$.
  Since also $\ve \ra v$ a.e.\ as $\eps = \eps_j \sea 0$ by \eqref{eq:eps_sea_0:v_l2},
  we conclude $\tilde v = v$ a.e.\ and thus \eqref{eq:strong_grad_conv:v}.
\end{proof}

The convergence of $∇\ln(\we+1)$ in \eqref{eq:eps_sea_0:nabla_ln_w_l2} seems much too weak for a treatment of $∇\we$. 
If, however, the latter term is combined with some cutoff, these two a priori different modes of convergence become indistinguishable.
We make this observation precise in the following form allowing for rather direct application to the equations at hand. 
\begin{lem}\label{lm:boundedconvergence}
  Let $(u, v, w)$ and $(\eps_j)_{j \in \N}$ be as given by Lemma~\ref{lm:eps_sea_0} and Lemma~\ref{lm:strong_grad_conv}, respectively.
  Moreover, let $T>0$, $ζ\in C^0_c([0,∞)^2)$ and $φ\in C^\infty(\Ombar \times [0, T])$. Then
  \begin{alignat}{2}
   ζ(\ve,\we)∇\we\sqrt{|φ|} &\wto ζ(v,w)∇w\sqrt{|φ|} &&\qquad \text{in } L^2(\Om\times(0,T))\label{conv:naw}\\
   \we ζ(\ve,\we)∇(\ue\ve)\cdot ∇\we &\wto wζ(v,w)∇(uv)\cdot ∇w &&\qquad \text{in } L^1(\Om\times(0,T))\label{conv:alarmterm}
  \end{alignat}
  as $\eps = \eps_j \sea 0$. 
\end{lem}
\begin{proof}
  We first note that the families $((\we+1)ζ(\ve,\we)\sqrt{|φ|})_{\eps \in (0, 1)}$, $(\ve\we(\we+1)ζ(\ve,\we))_{\eps \in (0, 1)}$ and $(\ue\we(\we+1)ζ(\ve,\we))_{\eps \in (0, 1)}$ are bounded due to compactness of the support of $ζ$ and Lemma~\ref{lm:ue_bdds} and, as $\eps = \eps_j \sea 0$, converge pointwise to their counterparts without $\eps$ as shown by \eqref{eq:eps_sea_0:delta_u_l2}, \eqref{eq:eps_sea_0:v_l2} and \eqref{eq:eps_sea_0:w_l1}.
  Combined with the identity $\frac{\nabla \we}{\we + 1}=\nabla \ln(\we + 1)$ and \eqref{eq:eps_sea_0:nabla_ln_w_l2}, this first shows \eqref{conv:naw}.
  From Lemma~\ref{lm:strong_grad_conv} we moreover have that $\nabla \ue \to \nabla u$ and $\mathds 1_{\{\ve \le k\}}\nabla \ve \to \mathds 1_{\{v \le k\}}\nabla v$ in $L^2(\Om\times(0,T))$ as $\eps = \eps_j \sea 0$ for all $k \in \N$.
  Choosing $k \in \N$ so large that $\supp ζ\subset [0,k)^2$ and again relying on \eqref{eq:eps_sea_0:nabla_ln_w_l2}, we conclude \eqref{conv:alarmterm}.
\end{proof}

\subsection{Proof of Theorem~\ref{th:xi>0}}\label{sec:xi>0:proof_thm}
At last, we show that Lemmata~\ref{lm:eps_sea_0}--\ref{lm:boundedconvergence} allow us to transfer sufficiently many solution properties
from $(\ue, \ve, \we)$ to the limit $(u, v, w)$ obtained in Lemma~\ref{lm:eps_sea_0};
that is, we show that the later is a global generalized solution of \eqref{system}.

\begin{lem}\label{lm:uvw_gen_sol}
  The triple $(u, v, w)$ constructed in Lemma~\ref{lm:eps_sea_0} is a global generalized solution of \eqref{system} in the sense of Definition~\ref{def:sol_concept}.
\end{lem}
\begin{proof}
  We let $(\eps_j)_{j \in \N}$ be as given by Lemma~\ref{lm:strong_grad_conv}. 
  The required regularity and nonnegativity of $(u, v, w)$ already follow from Lemma~\ref{lm:eps_sea_0} 
  and the estimate $\mathds 1_{\{w \le k\}} |\nabla w| = (w+1) \mathds 1_{\{w \le k\}} |\nabla \ln(w+1)| \le (k+1) |\nabla \ln(w+1)|$, valid for all $k \in \N$.
  Moreover, \eqref{eq:sol_concept:u_weak} and \eqref{eq:sol_concept:v_weak} rapidly result
  from the fact that $\ue$ and $\ve$ are weak solutions of the respective subproblems in \eqref{prob:approx} for all $\eps \in (0, 1)$
  combined with the convergence properties asserted in \eqref{eq:conv_init}, Lemma~\ref{lm:eps_sea_0} and Lemma~\ref{lm:strong_grad_conv}.
  For instance, \eqref{eq:eps_sea_0:v_l2} and \eqref{eq:strong_grad_conv:u} imply $\ve \nabla \ue \to v \nabla u$ in $L_{\loc}^1(\Ombarinf)$ as $\eps = \eps_j \sea 0$,
  while $g(\ue, \ve, \we) \rh g(u, v, w)$ in $L_{\loc}^1(\Ombarinf)$ as $\eps = \eps_j \sea 0$ follows from \eqref{eq:eps_sea_0:u_l2}, \eqref{eq:eps_sea_0:v_l2} and \eqref{eq:eps_sea_0:w_l2}.

  As to \eqref{eq:sol_concept:vw_phi_supersol},
  we fix $\phi \in C^\infty([0, \infty)^2)$ with $D \phi \in C_c^\infty([0, \infty)^2)$ and $\phi_{ww} \le 0$ in $[0, \infty)^2$ as well as a nonnegative $\varphi \in C_c^\infty(\Ombarinf)$
  and choose $T > 0$ such that $\supp \varphi \subset \Ombar \times [0, T]$.
  Testing the second and third equation in \eqref{prob:approx} with $\phi_v(\ve, \we) \varphi$ and $\phi_w(\ve, \we) \varphi$, respectively,
  shows that an analogue of \eqref{eq:sol_concept:vw_phi_supersol} (even with equality) also holds on the $\eps$-level. 
  As the norm is weakly lower semicontinuous, \eqref{conv:naw} (applied with $ζ=\sqrt{\phi_{ww}}$) implies
  \begin{align*}
        \liminf_{j \to \infty} \intntom \phi_{ww}(\vej, \wej) |\nabla \wej|^2 \varphi
    \ge \intntom \phi_{ww}(v, w) |\nabla w|^2 \varphi.
  \end{align*}
  Moreover, according to \eqref{conv:alarmterm}, 
  \begin{align*}
        \lim_{j \to \infty} \intntom \wej \phi_{ww}(\vej, \wej) \nabla (\uej \vej) \cdot \nabla \wej \varphi
    =   \intntom \phi_{ww}(v, w) w \nabla (u v) \cdot \nabla w \varphi.
  \end{align*}
  All other terms converge to their respective counterparts in \eqref{eq:sol_concept:vw_phi_supersol}
  due to similar or straightforward applications of \eqref{eq:conv_init}, Lemma~\ref{lm:eps_sea_0} and Lemma~\ref{lm:strong_grad_conv}.

  Finally, by \eqref{eq:eps_sea_0:w_l1} and possibly after switching to a subsequence,
  there exists a null set $N \subset (0, \infty)$ such that $\intom \wej(\cdot, t) \to \intom w(\cdot, t)$ as $j \to \infty$ for all $t \in (0, \infty) \setminus N$.
  We now claim that \eqref{eq:sol_def:mass_ineq} holds with this choice of $N$.
  Indeed, for any $t \in (0, \infty) \setminus N$, \eqref{eq:eps_sea_0:u_l2}, \eqref{eq:eps_sea_0:v_l2} and \eqref{eq:eps_sea_0:w_l2}
  imply $\we(\lambda_3 + b_2 \ue + b_3 \ve) \rh w(\lambda_3 + b_2 u + b_3 v)$ in $L^1(\Omega \times (0, t))$ as $\eps = \eps_j \sea 0$.
  As additionally \eqref{eq:eps_sea_0:w_l2} and the weak lower semicontinuity of the norm assert $\liminf_{j \to \infty} \intnstom \wej^2 \ge \intnstom w^2$,
  combining these convergence properties, testing the third equation in \eqref{prob:approx} with $1$ and recalling \eqref{eq:conv_init} gives
  \begin{align*}
        \intnstom h(u, v, w)
    &=  \intnstom w(\lambda_3 - \mu_3 w + b_2 u + b_3 v)
     \ge \limsup_{j \to \infty} \intnstom h(\uej, \vej, \wej) \\
    &= \limsup_{j \to \infty} \left( \intom \wej(\cdot, t) - \intom \wnej \right)
    = \intom w(\cdot, t) - \intom w_0,
  \end{align*}
  that is, \eqref{eq:sol_def:mass_ineq} holds.
\end{proof}

\begin{proof}[Proof of Theorem~\ref{th:xi>0}]
  The statement is contained in Lemma~\ref{lm:uvw_gen_sol}.
\end{proof}


\end{document}